\documentclass[leqn,10pt]{amsart}
\usepackage[utf8]{inputenc}
\usepackage{a4}
\usepackage[safeinputenc=true,style=alphabetic,firstinits,maxnames=4,useprefix=true,backend=bibtex]{biblatex}

\bibliography{upcrossing}
\usepackage{amsmath,amssymb,amscd,bbm}
\usepackage{amsthm}
\usepackage{enumerate}
\usepackage[mathscr]{eucal}
\usepackage[all]{xy}

\newtheorem{thm}{Theorem}[section]

\newtheorem{lemma}[thm]{Lemma}

\theoremstyle{definition}

\newtheorem{exa}[thm]{Example}

\theoremstyle{remark}

\newtheorem{rem}[thm]{Remark}

\numberwithin{equation}{section}

\newcommand{\vanish}[1]{\relax}       
\def\qedsymbol{\hbox to 1ex{\llap{\rule{0.25pt}{1ex}}\rlap{\rule{1ex}{0.25pt}}\lower0.25pt\rlap{\raise1ex\rlap{\rule{1ex}{0.25pt}}}\hskip1ex\llap{\rule{0.25pt}{1ex}}}}
\def\rlqed{\rlap{\rule{\hsize}{0pt}\kern-1ex\kern-1em\qed}} 

\newcounter{aufzi}
\newenvironment{aufzi}{\begin{list}{ {\upshape(\alph{aufzi})}}{
        \usecounter{aufzi}
        \topsep1ex
        \parsep0cm
        \itemsep0.8ex
        \leftmargin1cm
        \labelwidth0.5cm
        \labelsep0.3cm
}}
{\end{list}}

\newcounter{aufzii}

\newcounter{aufziii}
\newenvironment{aufziii}{\begin{list}{ {\upshape\arabic{aufziii})}}{
        \usecounter{aufziii}
        \topsep1ex
        \parsep0cm
        \itemsep0.8ex
        \leftmargin1cm
        \labelwidth0.5cm
        \labelsep0.3cm
}}
{\end{list}}



\newcommand{\bbP}{\mathbb{P}}

\newcommand{\calC}{\mathcal{C}}

\newcommand{\calF}{\mathcal{F}}

\newcommand{\calU}{\mathcal{U}}
\newcommand{\calV}{\mathcal{V}}
\newcommand{\calW}{\mathcal{W}}

\newcommand{\calY}{\mathcal{Y}}
\newcommand{\calZ}{\mathcal{Z}}

\def\ue{\mathrm{e}}

\renewcommand{\ue}{\mathrm{e}}    

\newcommand{\R}{\mathbb{R}}     
\newcommand{\N}{\mathbb{N}}
\newcommand{\Z}{\mathbb{Z}}

\newcommand{\Bigcap}[2][\relax]{%
 \ifx#1\relax \bigcap_{#2}
 \else \bigcap^{#1}_{#2}
 \fi}
\newcommand{\Bigcup}[2][\relax]{%
 \ifx#1\relax \bigcup_{#2}
 \else \bigcup^{#1}_{#2}
 \fi}

\def\fact#1#2{#1/#2}
\def\tfact#1#2{#1/#2}

\def\fact#1#2{{\raise0.2em\hbox{$#1$}\kern-0.2em/\kern-0.1em\lower0.2em\hbox{$#2$}}}
\def\tfact#1#2{{\raise0.1em\hbox{\small$#1$}\kern-0.1em/\kern-0.1em\lower0.1em\hbox{\small$#2$}}}

\newcommand{\norm}[2][\relax]{
   \ifx#1\relax \ensuremath{\left\Vert#2\right\Vert}
   \else \ensuremath{\left\Vert#2\right\Vert_{#1}}
   \fi}

\newcommand{\Bnorm}[2][\relax]{
   \ifx#1\relax \ensuremath{\Bigl\Vert#2\Bigr\Vert}
   \else \ensuremath{\Bigl\Vert#2\Bigr\Vert_{#1}}
   \fi}

\makeatletter
\newcommand{\tdprod}[2]{\ensuremath{%
  \setbox0=\hbox{\ensuremath{\langle#1,#2 \rangle}}
  \dimen@\ht0
  \advance\dimen@ by \dp0 (#1\rule[-\dp0]{0pt}{\dimen@}\,|#2\hspace{1pt})}}
\newcommand{\dprod}[2]{\ensuremath{%
  \setbox0=\hbox{\ensuremath{\left\langle#1,#2\right\rangle}}
  \dimen@\ht0
  \advance\dimen@ by \dp0 \left\langle\left.#1\rule[-\dp0]{0pt}{\dimen@}\,\right|#2\hspace{1pt}\right\rangle}}

\newcommand{\bdprod}[2]{\ensuremath{%
  \setbox0=\hbox{\ensuremath{\bigl\langle#1,#2\bigr\rangle}}
  \dimen@\ht0
  \advance\dimen@ by \dp0 \bigl\langle#1\bigl|\rule[-\dp0]{0pt}{\dimen@}\bigr.#2\hspace{1pt}\bigr\rangle}}
\newcommand{\Bdprod}[2]{\ensuremath{%
  \setbox0=\hbox{\ensuremath{\Bigl\langle#1,#2\Bigr\rangle}}
  \dimen@\ht0
  \advance\dimen@ by \dp0 \Bigl\langle#1\Bigl|\rule[-\dp0]{0pt}{\dimen@}\Bigr.#2\hspace{1pt}\Bigr\rangle}}
\newcommand{\tsprod}[2]{\ensuremath{%
  \setbox0=\hbox{\ensuremath{(#1,#2)}}
  \dimen@\ht0
  \advance\dimen@ by \dp0 (#1\rule[-\dp0]{0pt}{\dimen@}\,|#2\hspace{1pt})}}
\newcommand{\sprod}[2]{\ensuremath{%
  \setbox0=\hbox{\ensuremath{\left(#1,#2\right)}}
  \dimen@\ht0
  \advance\dimen@ by \dp0 \left(\left.#1\rule[-\dp0]{0pt}{\dimen@}\,\right|#2\hspace{1pt}\right)}}
\newcommand{\bsprod}[2]{\ensuremath{%
  \setbox0=\hbox{\ensuremath{\bigl(#1,#2\bigr)}}
  \dimen@\ht0
  \advance\dimen@ by \dp0 \bigl(#1\bigl|\rule[-\dp0]{0pt}{\dimen@}\bigr.#2\hspace{1pt}\bigr)}}
\newcommand{\Bsprod}[2]{\ensuremath{%
  \setbox0=\hbox{\ensuremath{\Bigl(#1,#2\Bigr)}}
  \dimen@\ht0
  \advance\dimen@ by \dp0 \Bigl(#1\Bigl|\rule[-\dp0]{0pt}{\dimen@}\Bigr.#2\hspace{1pt}\Bigr)}}
\makeatother

\newcommand{\Cc}{\mathrm{C}_\mathrm{c}}

\newcommand{\Ell}[2][\relax]{
   \ifx#1\relax \mathrm{L}^{\mathrm{#2}}
   \else \mathrm{L}^{\mathrm{#2}}_{\mathrm{#1}}
   \fi}
\renewcommand{\Ell}[2][\relax]{
   \ifx#1\relax \mathrm{L}^{\!#2}
   \else \mathrm{L}^{\!#2}_{\mathrm{#1}}
   \fi}

\newcommand{\Wee}[2][\relax]{
   \ifx#1\relax \mathrm{W}^{\mathrm{#2}}
   \else \mathrm{W}^{\mathrm{#2}}_{\mathrm{#1}}
   \fi}
\newcommand{\Har}[2][\relax]{
   \ifx#1\relax \mathsf{H}^{\mathsf{#2}}
   \else   \mathsf{H}^{\mathsf{#2}}_{\mathrm{#1}}
   \fi}

\def\prX{\mathrm X}    
\def\prY{\mathrm Y}

\def\rlqed{\rlap{\rule{\hsize}{0pt}\kern-1ex\kern-1em\qed}}

\makeatletter

\def\maketag@@@@@#1{\llap{\hbox to\hsize{\m@th\normalfont#1}}%
\gdef\tagform@##1{\maketag@@@{(\ignorespaces##1\unskip\@@italiccorr)}}}

\def\eqtext#1{\gdef\tagform@##1{\maketag@@@@@{\ignorespaces##1\unskip\@@italiccorr\hfill}}\tag{#1}}%
\def\reqtext#1{\gdef\tagform@##1{\maketag@@@@@{\hfill\ignorespaces##1\unskip\@@italiccorr}}\tag{#1}}%
\def\leqtext#1{\gdef\tagform@##1{\maketag@@@@@{\ignorespaces##1\unskip\@@italiccorr}}\tag{#1}}%
\makeatother

\newcommand{\UT}[2]{\mathsf{UT}_{#1}(#2)}

\newcommand{\cntm}[1]{\left|#1\right|}
\newcommand{\Rp}{\R_{\geq 0}}
\newcommand{\Rps}{\R_{>0}}
\newcommand{\Zp}{\Z_{\geq 0}}
\newcommand{\bl}{\mathrm B}
\newcommand{\bdr}[1]{\partial_{#1}}
\newcommand{\intr}[1]{\mathrm{int}_{#1}}
\newcommand{\rad}{\mathrm{rad}}
\newcommand{\hght}{\mathrm{ht}}
\newcommand{\core}{\mathrm{core}}

\newcommand{\flc}[3]{\mathcal{F}_{(#1,#2)}^{#3}}

\newcommand{\kint}[2]{\mathrm{int}_{#1}(#2)}

\newcommand{\indi}[1]{\mathbf{1}_{#1}}

\date{\today}

\begin{document}

\title[Upcrossing theorems]{Hochman's upcrossing theorem\\
for groups of polynomial growth.}

\author[N. Moriakov]{Nikita Moriakov}
\address{Delft Institute of Applied Mathematics, Delft University of Technology,
P.O. Box 5031, 2600 GA Delft, The Netherlands}

\email{n.moriakov@tudelft.nl}

\subjclass{Primary 28D15, 60G10, 60G17}
\renewcommand{\subjclassname}{\textup{2000} Mathematics Subject
    Classification}

\date{\today}

\begin{abstract}
Consider a stochastic process $(S_{[a_i,b_i]})_{[a_i,b_i] \subset \mathbb{N}}$, which is indexed
by the collection of all nonempty intervals $[a_i,b_i] \subset \mathbb{N}$ and which is
stationary under translations of the intervals. It was shown by M. Hochman that,
for any $k \geq 1$ and any interval $(\alpha,\beta) \subset
\mathbb{R}$, one can give an `almost-exponential' bound on the
size of the set where the associated process $(S_{[1,n]})_{n \geq 1}$ has at
least $k$ fluctuations over $(\alpha,\beta)$. It was also noticed that
a similar techniques can be applied in $\mathbb{Z}^d$ case. In this article we extend
Hochman's upcrossing theorem to groups of polynomial growth. 
\end{abstract}

\maketitle

\section{Introduction}
Given an integer $n \in \Zp$ and some numbers $\alpha,\beta \in \R$
such that $\alpha<\beta$, a sequence of real numbers $(a_i)_{i = 1}^k$
is said to have \textbf{at
least $n$ upcrossings} across the interval $(\alpha,\beta)$ if there
are indexes $1 \leq i_1 < j_1 < i_2 < j_2 < \dots < i_n < j_n
\leq k$ such
that
\begin{aufziii}
\item $a_{i_l} < \alpha$ for all indexes $l$;
\item $a_{j_l} > \beta$ for all indexes $l$.
\end{aufziii} 
If $(a_i)_{i \geq 1}$ is an infinite sequence of real numbers, we use the
same terminology and say
that $(a_i)_{i \geq 1}$ has at least $n$ upcrossings across the
interval $(\alpha,\beta)$ if some initial segment $(a_i)_{i=1}^k$ of
the sequence has at least $n$ upcrossings across $(\alpha,\beta)$. To
simplify the notation we
denote the sets of real-valued sequences having at least $n$
upcrossings across an interval $(\alpha, \beta)$ by
$\calF_{(\alpha,\beta)}^n$.

The purpose of this article is to generalize the following result of
M. Hochman from \cite{hochman2009}. Consider a real-valued stochastic process
$(S_{[i,j]})_{1\leq i \leq j}$, indexed by the collection of all
integer intervals $[i,j] \subset \N$. Suppose that this process is
stationary, i.e.,
\[
(S_{[i,j]})_{1 \leq i \leq j} = (S_{[i+k,j+k]})_{1\leq i \leq j} \quad
\text{ for all } k \in \N
\]
in distribution. For a measure space $(\prY,\calC,\nu)$, a measurable subset $I
\subseteq Y$ and a number $\delta>0$
we say that a collection of measurable sets $I_1,\dots,I_k$
\textbf{$\delta$-fills} $I$ if $I_i \subseteq I$ for all indices $i$ and
$\nu\left(I \setminus \bigcup\limits_{i} I_i \right) < \delta \nu(I)$. When
working with subsets of $\N$, or subsets of groups of polynomial
growth, we always use the counting measure. The
result of Hochman states the following. Let $(\alpha,\beta) \subset
\R$ be some interval and
$0< \delta<\frac 1 4$ be some constant. Then there exist constants $c>0$ and $\rho \in
(0,1)$, depending only on $\alpha,\beta,\delta$, such that for every stationary process $(S_{[i,j]})_{1\leq i
  \leq j}$ and all $k \geq 1$
\begin{align*}
&\bbP(\{ x: (S_{[1,i]}(x))_{i \geq 1} \in \flc{\alpha}{\beta}{k}  \})
  \leq \\
 &\leq c \rho^k + \bbP \left(\left\{ x: 
\begin{array}{c}
\exists n>k \text{ s.t. }
  S_{[1,n]}(x)>\beta \text{ and } \bl(n) \text{ can be } \delta-\text{filled}\\
\text{by disjoint intervals } V_1,\dots,V_m \text{ s.t. } \forall i \
  S_{V_i}(x)<\alpha
\end{array} \right\} \right).
\end{align*}

We generalize this result as follows. Suppose that $\Gamma$ is an
discrete group of polynomial growth endowed with the word
metric $d$, see Section \ref{ss.grouppolgr} for the definitions. We consider stationary
processes $(S_{\bl(g,r)})$, indexed by the collection of
all balls 
\[
\{ \bl(g,r): g \in \Gamma, r \in \Zp \},
\]
which we view plainly as subsets of $\Gamma$ for the moment. We denote
by $\bl(n)$ the ball $\bl(\ue,n)$ of radius $n$ around the neutral
element $\ue \in \Gamma$. Stationarity of the process means that
\[
(S_{\bl(g,r)})_{g,r} = (S_{\bl(g+h,r)})_{g,r} \quad \text{ for all } h \in \Gamma
\]
in distribution. The main result of this article is the following theorem.
\begin{thm}
\label{t.hut}
Let $\Gamma$ be a group of polynomial growth. For all intervals
$(\alpha,\beta) \subset \R$ and all $\delta>0$ small enough there
exist constants $c>0$ and $\rho \in (0,1)$ such that the following
holds. If $(S_{\bl(g,r)})$ is a stationary process, then for
all $k \geq 1$
\begin{align*}
&\bbP(\{ x: (S_{\bl(i)}(x))_{i \geq 1} \in \flc{\alpha}{\beta}{k}  \})
  \leq \\
 &\leq c \rho^k + \bbP \left(\left\{ x: 
\begin{array}{c}
\exists n>k \text{ s.t. }
  S_{\bl(n)}(x)>\beta \text{ and } \bl(n) \text{ can be } \delta-\text{filled}\\
\text{by disjoint balls } V_1,\dots,V_m \text{ s.t. } \forall i \
  S_{V_i}(x)<\alpha
\end{array} \right\} \right)
\end{align*}
\end{thm}

The paper is structured as follows. In Section \ref{ss.measpolgr} we
introduce the notion of a metric measure space of exact polynomial
volume growth, provide some examples and prove some elementary properties of
these spaces. Section \ref{ss.grouppolgr} is devoted to discrete
groups of polynomial growth, which become metric measure spaces of
uniformly exact polynomial volume growth when endowed with the
counting measure and the word metric. In Section \ref{s.covlem} we
derive some covering lemmas on metric measure spaces of exact
polynomial volume growth, including the `Effective Vitali Covering'
(Theorem \ref{t.evc}) and the generalization of Hochman's `tower
sandwich' lemma (Theorem \ref{t.expgr}), which are the main technical
tools in this article.

The upcrossing theorem for groups of polynomial growth is proved in
Section \ref{ss.hochman}. As one of the applications, we show in Section \ref{ss.kingman} how Hochman's proof of exponential
upcrossing inequality for Kingman's subadditive ergodic theorem
can be obtained for discrete groups of polynomial growth using Theorem \ref{t.hut}.

\section{Preliminaries}
\subsection{Metric Measure Spaces of Exact Polynomial Volume Growth}
\label{ss.measpolgr}
Let $\prX=(X,d)$ be a metric space. A closed ball with radius $r \in \Rp$
and center $x \in X$ will be denoted by $\bl(x,r)$. Let $\mu$ be a Borel measure on
$X$. A pair $(\prX,\mu)$ will be called a \textbf{metric measure
  space}. We say that the measure $\mu$ on $X$ is \textbf{strictly
  positive} if for all points $x \in X$ and all radii $r \in \Rps$
\begin{equation*}
0<\mu(\bl(x,r))<\infty.
\end{equation*}
\noindent A metric measure space $(\prX,\mu)$ with a strictly positive measure $\mu$ will
be called a \textbf{space of exact polynomial volume growth} if there are
constants $c_X \in \Rps$ and $q \in \Rp$ such that
\begin{equation}
\label{eq.expolgr}
\lim\limits_{t \to \infty} \frac{\mu(\bl(x,t))}{c_X t^q} = 1 \quad \text{
  for all } x \in X.
\end{equation}
The fact that $\mu$ is strictly positive implies that the
constants $c_X$ and $q$ are uniquely determined. We call $q$ the
\textbf{degree of polynomial volume growth}. If the limit in Equation
\eqref{eq.expolgr} converges uniformly on a subset $W \subseteq X$, we say that
$(\prX,\mu)$ has \textbf{uniformly exact polynomial volume growth on $W$}. If
the limit converges uniformly on the whole space $X$, the space
$(\prX,\mu)$ will be called a \textbf{space of uniformly exact
  polynomial volume growth}. 

Many interesting examples of metric measure spaces with (uniformly)
exact polynomial volume growth are given by groups endowed with a translation-invariant
metric and a Haar measure. We begin with some basic examples.

\begin{exa}
\label{ex.rd}
Let  $\prX:=(\R^k,d)$ be the $k$-dimensional Eucledean space, which we
endow with the $\ell^2$-norm $\| \cdot \|_2$
and the associated translation-invariant metric $d$ given by
\[
d(x,y):=\| x - y \|_{2}.
\]
It is easy to see that the Lebesgue measure $\mu$ on $\R^k$ is doubling,
and, furthermore, that $(\prX,\mu)$ is a metric measure space of
uniformly exact polynomial volume growth of degree $k$.
\end{exa}

\begin{exa}
\label{ex.heisenberg}
Consider the real Heisenberg group $\UT{3}{\R}$. By definition,
\begin{equation*}
\UT{3}{\R} := \left\{ \left( \begin{matrix}
        1 & a & c \\
        0 & 1 & b \\
        0 & 0 & 1
      \end{matrix} \right): a,b,c \in \R  \right\}.
\end{equation*}
To simplify the notation, we will denote a matrix
\begin{equation*}
\left( \begin{matrix}
        1 & a & c \\
        0 & 1 & b \\
        0 & 0 & 1
      \end{matrix} \right) \in \UT{3}{\R}
\end{equation*}
by the corresponding triple $(a,b,c)$ of its
entries. The Lebesgue measure on $\R^3$ becomes the Haar measure $\mu$ on $\UT{3}{\R}$
under this identification, i.e., for all compactly supported
continuous functions $f \in \Cc(\UT{3}{\R})$ we
have
\[
\int\limits_{\UT{3}{\R}} f d \mu = \int\limits_{\R}
\int\limits_{\R} \int\limits_{\R} f(a,b,c) da db dc.
\]
The \emph{Kor\'{a}nyi norm} of an element $(a,b,c) \in \UT{3}{\R}$ is defined by
\[
\| (a,b,c) \|_K := ((a^2 + b^2)^2 + c^2 )^{1/4},
\]
and the associated left-invariant metric is given by
\[
d_K(x,y):=\| x^{-1} y\|_K \quad (x,y \in \UT{3}{\R}).
\]
A direct computation shows that, with respect to this metric, $((\UT{3}{\R},d_K),\mu)$ is a metric measure space of uniformly exact
 polynomial volume growth of degree $4$.
\end{exa}

\begin{rem}
\label{r.liegr}
Furthermore, one can generalize these examples and show that connected nilpotent homogeneous groups have
uniformly exact polynomial volume growth as well. For more details we refer
to \cite{nevo2006}. For the definition of a homogeneous group we refer
to \cite{folste1982}. We will not discuss these results in the article, since our
main object of interest in this paper are (discrete) \emph{groups of polynomial growth}, which become metric measure spaces with
uniformly exact polynomial volume growth when endowed with a counting measure
and a \emph{word metric}. We introduce groups of polynomial
growth in the next section (Section \ref{ss.grouppolgr}).  
\end{rem}

\begin{rem}
\label{r.nondoubling}
A metric measure space $(\prX,\mu)$ with a strictly positive measure
$\mu$ is called \textbf{doubling} if
there exists a constant $C>0$ such that for all $x \in X, r \in \Rp$
we have
\[
\mu(\bl(x, 2r)) \leq C \mu(\bl(x,r)).
\]
The concept of a doubling metric measure space is standard in
analysis. One may wonder if a metric measure space of uniformly exact
polynomial growth is necessarily doubling - and the answer in general is `no'. The reason is
that the definition of uniformly exact polynomial growth only controls
the size of the balls `at large scale', while the doubling condition
at small scales can still be violated.
\end{rem}

A natural question is if all metric measure spaces of exact polynomial
volume growth have uniformly exact polynomial volume growth. The
answer is `no': it is easy to construct a metric $d'$ on $\R$ such
that
\begin{aufzi}
\item $d'$ is equivalent to $d$ in the sense of topology (but
  \emph{not} bi-Lipschitz equivalent);
\item $((\R,d'),\mu)$ is a metric measure space of exact, but \emph{not}
  uniformly exact, polynomial
  growth of degree $1$.
\end{aufzi}
However, the following lemma shows that the volume growth is always uniformly exact
on bounded subsets of metric measure spaces of exact polynomial growth.
\begin{lemma}
\label{l.exgronballs}
Let $(\prX, \mu)$ be a metric measure space of exact polynomial
growth. Let $W=\bl(x,r) \subseteq X$ be a ball around $x \in X$. Then
$(\prX,\mu)$ has uniformly exact polynomial growth on $W$.
\end{lemma}
\begin{proof}
Fix a point $y \in W$. Then the proof easily follows from the
definition of exact polynomial growth combined with the observations that
\[
\bl(x,s-r) \subseteq \bl(y,s) \quad \text{ for } s>r
\]
and
\[
\bl(y,s) \subseteq \bl(x,r+s) \quad \text{ for } s>r.
\]
\end{proof}

If $t \in \Rp$ is
some number and $W=\bl(y,s) \subseteq X$ is some ball, we will denote by $t \cdot
W$ the ball $\bl(y,ts)$. For a ball $W = \bl(y,s)$ in a metric measure space $(\prX,\mu)$ and
a number $\delta \in (0,1)$ we
define the \textbf{$\delta$-interior} of $W$ as the ball
\[
\intr{\delta}(W):=\bl(y,(1-\delta)s),
\]
and the \textbf{$\delta$-boundary} of $W$ as the set
\[
\bdr{\delta}(W):=W \setminus \intr{\delta}(W).
\]
For $\delta \in \Rp$ the \textbf{$\delta$-expansion} of $W$ is the ball
\[
W^{\delta}:= \bl(y,(1+\delta) s).
\]
In principle, these objects do depend on a concrete representation
of $W$ as a ball $\bl(y,s)$, and not just on $W$ \emph{as a subset} of
$X$. Thus, whenever we are talking about balls or collections of
balls, we always assume that concrete centers and radii are provided,
and this will always be the case in applications. Given a collection $\calC$ of balls in $(\prX,\mu)$ and
a number $\delta \in (0,1)$, we define the $\delta$-interior of
$\calC$ as
\[
\intr{\delta}(\calC):=\bigcup\limits_{W \in \calC} \intr{\delta}(W),
\]
and the $\delta$-boundary of $\calC$ as
\[
\bdr{\delta}(\calC):=\bigcup\limits_{W \in \calC} \bdr{\delta}(W).
\]
For $\delta \in \Rp$ the $\delta$-expansion of $\calC$ is
\[
\calC^{\delta}:=\bigcup\limits_{W \in \calC} W^{\delta}.
\]
Next, the \textbf{radius} of $\calC$ is defined as the infinum of radii of balls in $\calC$, i.e.,
\[
\rad(\calC):=\inf\{ r: \bl(x,r) \in \calC \},
\]
and the \textbf{core} of $\calC$ is defined as the set of centers
of balls in $\calC$, i.e.,
\[
\core(\calC):=\{ x: \bl(x,r) \in \calC \}.
\]
We will now prove a basic lemma, saying that in metric measure spaces
of exact polynomial volume growth we can control the size of the
$\delta$-boundary of collections of balls with sufficiently large radius.
\begin{lemma}[`Controlling the size of $\delta$-boundary']
\label{l.smallbdr}
Let $(\prX,\mu)$ be a metric measure space of exact polynomial growth
of degree $q>0$ and $X_0 \subseteq X$ be a bounded set. There exists a number
$\varepsilon_0 \in (0,1)$ such that for every $\varepsilon \in (0,\varepsilon_0)$ there
exists a number $s = s(\varepsilon,X_0)$ so that the
following assertion holds. If $\calC$ is a countable collection of
disjoint balls in $X$ such that
\[
\rad(\calC)  \geq s
\]
and
\[
\core(\calC) \subseteq X_0,
\]
then for all $0< \delta \leq \min(\frac{\varepsilon}{4 q}, \frac 1 2 )$ we have
\begin{equation}
\label{eq.smallbdr1}
\mu(\intr{\delta}(\calC)) \geq (1-\varepsilon)\mu \left( \bigsqcup\limits_{U \in
\calC} U \right)
\end{equation}
and
\begin{equation}
\label{eq.smallbdr2}
\mu(\bdr{\delta}(\calC)) \leq \varepsilon \mu \left( \bigsqcup\limits_{U \in
\calC} U \right)
\end{equation}
\end{lemma}
\begin{proof}
Let $\varepsilon_0>0$ be small enough so that for all
$0< y \leq \varepsilon_0$ we have
\[
\left( 1-\frac{y}{2q}\right)^q > 1-y.
\]
Let $z=z(\varepsilon) \in (0,1)$ be small enough so that for a given
$0<\varepsilon \leq \varepsilon_0$ we have
\[
\frac{1-z}{1+z}(1-\varepsilon/2)>(1-\varepsilon).
\]
Lemma
\ref{l.exgronballs} tells us that there exists large enough $s=s(\varepsilon,X_0)$ so that 
\[
1-z < \frac{\mu(\bl(x,t))}{c_X t^q} < 1+z \quad \text{ for
  all } x\in X_0, t \geq s/2,
\]
where the constant $c_X \in \Rps$ is given by the definition of
metric measure space of exact polynomial volume growth. Then, for all
$\delta \in (0, \min(\frac{\varepsilon}{4 q}, \frac 1 2 ) ]$, $x \in X_0$ and $t>s$ we have
\[
\frac{1-z}{1+z}(1-\delta)^q < \frac{\mu(\bl(x,(1-\delta)t))}{\mu(\bl(x,t))}.
\]
Since
\[
\frac{1-z}{1+z}(1-\delta)^q>1-\varepsilon
\] 
for all $0< \delta \leq \min(\frac{\varepsilon}{4 q}, \frac 1 2 )$, this completes the proof.
\end{proof}

As for $\delta$-expansions, we state the following lemma.
\begin{lemma}[`Controlling the size of $\delta$-expansion']
\label{l.smallexp}
Let $(\prX,\mu)$ be a metric measure space of exact polynomial growth
of degree $q>0$ and let $X_0 \subseteq X$ be a bounded set. There
exists a number $\varepsilon_0 \in (0,1)$ such that for every
$\varepsilon \in (0,\varepsilon_0)$ there
exist a number $s = s(\varepsilon,X_0)$ such that the following assertion holds. If $\calC$ is a countable collection of
disjoint balls in $X$ such that
\[
\rad(\calC)  \geq s
\]
and
\[
\core(\calC) \subseteq X_0,
\]
then for all $\delta \in (0, \frac{\varepsilon}{4 q} )$ we have
\begin{equation}
\label{eq.smallexp1}
\mu \left( \bigsqcup\limits_{U \in
\calC} U \right) \geq (1-\varepsilon) \mu \left( \bigcup\limits_{U \in
\calC} U^{\delta} \right).
\end{equation}
\end{lemma}
The proof is almost identical to the proof of Lemma \ref{l.smallbdr},
so we leave it to the reader. The following lemma, whose proof is also
straightforward, shows that for metric measure spaces of positive
exact polynomial growth there \emph{exist} sufficiently distant points
in balls $\bl(x,r)$ as $r$ is large enough. 
\begin{lemma}
\label{l.distpoint}
Let $(\prX,\mu)$ be a metric measure space of exact polynomial growth
of degree $q>0$ and let $X_0 \subseteq X$ be a bounded set. There
exists a number $s_* = s_*(X_0)$ such that the following assertion
holds. If $\bl(x,r)$ is a ball of radius $r \geq s_*$ with center $x
\in X_0$, then there exist points $y_1, y_2 \in \bl(x,r)$ such that
$d(y_1,y_2)> \frac{2 r} 3$.
\end{lemma}

Finally, we state the last lemma of the section. The proof is
analogous to the proofs of the previous results, and we leave it to
the reader as well.
\begin{lemma}
\label{l.measvsrad}
Let $(\prX,\mu)$ be a metric measure space of exact polynomial growth
of degree $q>0$ and let $X_0 \subseteq X$ be a bounded set. There
exists a number $s_! = s_!(X_0)$ such that the following assertion
holds. If $\bl(x,r_1), \bl(x,r_2)$ are two balls with $x \in X_0$ such
that $r_1,r_2 \geq s_!$,
then
\[
\frac{\mu(\bl(x,r_1))}{\mu(\bl(x,r_2))} > 1+t \quad \text{ for some } t>0
\]
implies that
\[
\frac{r_1}{r_2} > \left( \frac{1+t}{1+2t/3}\right)^{1/q}.
\]
\end{lemma}

\begin{rem}
\label{r.smallbdrunexgr}
Suppose that $(\prX,\mu)$ is a metric measure space of
\emph{uniformly} exact polynomial volume growth of degree $q>0$. Then it is clear that one
can remove the assumption that
\[
\core(\calC) \subseteq X_0
\]
for a bounded set $X_0$ in Lemmas \ref{l.smallexp} and
\ref{l.smallbdr}. Furthermore, the numbers $s$ depend on $\varepsilon$
and space $(\prX,\mu)$ only. Similarly, the assumption that
\[
x \in X_0
\]
for a bounded set $X_0$ in Lemmas \ref{l.distpoint},
\ref{l.measvsrad} can be
removed and the numbers $s_*,s_!$ depend on the space $(\prX,\mu)$ only.
\end{rem}

\subsection{Groups of Polynomial Growth}
\label{ss.grouppolgr}
Let $\Gamma$ be a finitely generated discrete group and $\{
\gamma_1,\dots,\gamma_k\}$ be a fixed generating set. Each element
$\gamma \in \Gamma$ can be written as a product
$\gamma_{i_1}^{p_1} \gamma_{i_2}^{p_2} \dots \gamma_{i_l}^{p_l}$ for
some indexes $i_1,i_2,\dots,i_l \in 1,\dots,k$ and some integers $p_1,p_2,\dots,p_l \in \Z$. We define the \textbf{norm} of an element $\gamma
\in \Gamma$ by
\[
\| \gamma \|:=\inf\{ \sum\limits_{i=1}^l |p_i|: \gamma =
\gamma_{i_1}^{p_1} \gamma_{i_2}^{p_2} \dots \gamma_{i_l}^{p_l} \},
\]
where the infinum is taken over all representations of $\gamma$ as a
product of the generating elements. The norm $\| \cdot \|$ on $\Gamma$
can, in general, depend on the generating
set, but it is easy to show \cite[Corollary 6.4.2]{ceccherini2010} that
two different generating sets produce equivalent norms. We will always
say what generating set is used in the definition of a norm, but we will
omit an explicit reference to the generating set later on. 

The norm $\| \cdot \|$ yields a right invariant metric on $\Gamma$
defined by
\[
d_R(x,y):=\| x y^{-1}\| \quad \text{ for } x,y \in \Gamma,
\]
and a left invariant metric on $\Gamma$ defined by
\[
d_L(x,y):=\| x^{-1} y\| \quad \text{ for } x,y \in \Gamma,
\]
which we call the \textbf{word metrics}. The right invariance of $d_R$ means that
the right multiplication
\[
R_g: \Gamma \to \Gamma, \quad x \mapsto x g \quad ( x \in \Gamma)
\]
is an isometry for every $g \in \Gamma$ with respect to
$d_R$. Similarly, the left invariance of $d_L$ means that the left
multiplications are isometries with respect to $d_L$.
We let $d:=d_R$ and view
$\Gamma$ as a metric space with the metric $d$. For $x\in \Gamma$, $r
\in \Rp$ let
\[
\bl(x,r):=\{ y \in \Gamma: d(x,y) \leq r\}
\]
be the closed ball of radius $r$ with center $x$ with respect to $d$. Let $\ue \in \Gamma$ be the neutral element. It is clear that 
\[
\bl(n) = \{ y: d_R(\ue,y) \leq n\} = \{ y:
d_L(\ue,y) \leq n\} = \{ y: \| y \| \leq n \},
\]
i.e., the ball $\bl(n)$ is precisely the ball $\bl(\ue, n)$ with respect to the
left and the right word metric. Furthermore,
\[
\cntm{\bl(x,n)} = \cntm{\bl(n)}
\]
for all $x \in \Gamma$, $n \geq 0$.

We say that the group $\Gamma$ is of
\textbf{polynomial growth} if there are constants $C,d>0$ such that
for all $n \geq 1$ we have
\[
\cntm{\bl(n)} \leq C n^d.
\]
\begin{exa}
\label{ex.zdex}
Consider the group $\Z^d$ for $d \in \N$ and let $\gamma_1,\dots,\gamma_d \in \Z^d$ be the
standard basis elements of $\Z^d$. That is, $\gamma_i$ is defined by
\[
\gamma_i(j):=\delta_i^j \quad (j=1,\dots, d)
\] 
for all $i=1,\dots,d$. We consider the generating set given by elements $\sum\limits_{k \in I} (-1)^{\varepsilon_k}\gamma_k$ for all
subsets $I \subseteq [1,d]$ and all functions $\varepsilon_{\cdot} \in
\{ 0,1\}^I$. Then it is easy
to see by induction on dimension that $\bl(n) = [-n,\dots,n]^d$, hence
\[
\cntm{\bl(n)} = (2n+1)^d \quad \text{ for all } n \in \N
\]
with respect to this generating set, i.e., $\Z^d$ is a group of polynomial growth.
\end{exa}

Let $d \in \Zp$. We say that the group $\Gamma$ has \textbf{polynomial growth
of degree $d$} if there is a constant $C>0$ such that
\[
\frac 1 C n^d \leq \cntm{\bl(n)} \leq C n^d \quad \text{ for all } n \in \N.
\]
It was shown in \cite{bass1972} that, if $\Gamma$ is a finitely
generated nilpotent group, then $\Gamma$ has polynomial growth of some
degree $d \in \Zp$. Furthermore, one can show \cite[Proposition
6.6.6]{ceccherini2010} that if $\Gamma$ is a group and $\Gamma' \leq \Gamma$
is a finite index, finitely generated nilpotent subgroup, having
polynomial growth of degree $d \in \Zp$, then the group $\Gamma$ has
polynomial growth of degree $d$. The
converse is true as well: it was proved in \cite{gromov1981}
that, if $\Gamma$ is a group of polynomial growth, then there exists a
finite index, finitely generated nilpotent subgroup $\Gamma' \leq
\Gamma$. It follows that if $\Gamma$ is a group of polynomial growth
with the growth function $\gamma$, then there is a constant $C>0$ and
an integer $d\in \Zp$, called the \textbf{degree of polynomial growth}, such that
\[
\frac 1 C n^d \leq \cntm{\bl(n)} \leq C n^d \quad \text{ for all } n \in \N.
\]
An even stronger result was obtained in \cite{pansu1983}, where it is
shown that, if $\Gamma$ is a group of polynomial growth of degree $d
\in \Zp$, then the limit
\begin{equation}
\label{eq.pansu}
c_{\Gamma}:=\lim\limits_{n \to \infty} \frac{\cntm{\bl(n)}}{n^d}
\end{equation}
exists. It follows that a group of polynomial growth of degree $d \in
\Zp$ is a metric measure space of uniformly exact polynomial volume
growth of degree $d$. 

\section{Covering Lemmas}
\label{s.covlem}
First, we state the basic finitary Vitali covering lemma. The proof is
well-known, so we omit it.
\begin{lemma}
\label{l.finitvitcov}
Let $\calC:=\{B_1,\dots,B_n\}$ be a finite collection of balls in a metric
space $\prX$. Then there exists a disjoint subcollection $\calC_0
\subseteq \calC$ such that
\[
\bigcup\limits_{W \in \calC} W \subseteq \bigcup\limits_{W \in \calC_0}
3 \cdot W.
\]
\end{lemma}
\noindent As a consequence, we derive the following version of the Vitali lemma,
which we will later use in the proofs.
\begin{lemma}
\label{l.vitcov}
Let $(\prX,\mu)$ be a metric measure space of exact polynomial growth
of degree $q$
and let $X_0 \subseteq X$ be a bounded set. Then there exists a number
$s'=s'(X_0)>0$
such that the following assertion holds. If $\calC$ is a finite
collection of balls in $X$ such
that
\[
\rad(\calC)  \geq s'
\]
and
\[
\core(\calC) \subseteq X_0,
\]
then there exists a disjoint subcollection $\calC_0 \subseteq \calC$ such
that
\[
\mu\left( \bigcup\limits_{ V \in \calC_0} V \right) \geq \frac{1}{3^{q+1}}
\mu\left( \bigcup\limits_{V \in \calC} V \right)
\]
\end{lemma}
\begin{proof}
Let $s' \in \Rps$ be large enough such that for all balls $V=\bl(x,s)$
with $x \in X_0$ and $s \geq s'$ we have
\begin{equation}
\label{eq.3enlgr}
\mu(3 \cdot V) < 3^{q+1} \mu(V).
\end{equation}
Such $s'$ exists because the volume growth is uniformly polynomially
exact on $X_0$. Next, given a collection $\calC$ satisfying the
assertions of the theorem, we use Lemma \ref{l.finitvitcov} to obtain
a disjoint subcollection $\calC_0 \subseteq \calC$ such that 
\[
\bigcup\limits_{V \in \calC} V \subseteq \bigcup\limits_{V \in \calC_0}
3 \cdot V.
\]
Combining this with Equation \eqref{eq.3enlgr}, we obtain the
statement of the lemma.
\end{proof}
\begin{rem}
\label{r.smallbdrunexgr1}
Similar to Remark \ref{r.smallbdrunexgr} above we can note that for metric measure spaces of
\emph{uniformly} exact polynomial growth the assumption in Lemma
\ref{l.vitcov} that
\[
\core(\calC) \subseteq X_0
\]
for a bounded set $X_0$ can be dropped. In this case the number $s'$
does depends only on the space $(\prX,\mu)$. 
\end{rem}

When the original collection $\calC$ in Lemma \ref{l.vitcov}
satisfies certain additional assumptions, applying the lemma multiple
times can yield a disjoint subcollection which `almost covers'
$\bigcup\limits_{V \in \calC} V$. In order to make this precise, we
start with some definitions. Let $\prX$ be a metric space, $W \subseteq X$ be a
nonempty subset and $n \in \N$ be an integer. A \textbf{tower} $\calU$
over $W$ is a finite
collection of balls
\[
\calU = \{ U_i(x): x \in W, i = 1,\dots, n\}
\]
indexed by points $x \in W$ and integers $i=1,\dots,n$, which satisfies the following assertions:
\begin{aufziii}
\item $x \in U_i(x)$ for all $x \in W$ and all $i=1,\dots,n$;
\item $U_1(x) \subseteq U_2(x) \subseteq \dots \subseteq U_n(x)$ for all $x \in W$. 
\end{aufziii}
The number $n$ is called the \textbf{height} of a tower, and we denote
the height of a tower $\calU$ by $\hght(\calU)$. We stress that the
set $W$ might be infinite, but a tower $\calU$ over $W$ is always a
finite collection of balls. A tower $\calU$ over $W$ of height $n$ is
called \textbf{centered} if for every point $x \in W$ and every index $i=1,\dots,n$ the
ball $U_i(x)$ is of the form $\bl(x,r)$ for some $r \geq 0$. For a
tower $\calU$ of height $n$ and every index $i=1,\dots,n$ let
\[
\calU_i:=\{ U_i(x): U_i(x) \in \calU\}
\]
be the $i$-th level of the tower $\calU$ and let
\[
\calU_{\leq i}:=\bigcup\limits_{j=1}^i \calU_j,
\]
\[
\calU_{\geq i}:=\bigcup\limits_{j=i}^n \calU_j
\]
be the towers obtained from the tower $\calU$ by taking the first $i$
levels and the last $n-i+1$ levels respectively. The balls in $\calU_i$ will
be call \textbf{$i$-th level balls}. Given a
number $\delta \in \Rp$, a tower $\calU$ over $W \subseteq X$ of
height $n \in \N$ is called
\textbf{$\delta$-expanding} if
\begin{equation}
\label{eq.exptowdef}
U_i(x)^{\delta} \subseteq U_{i+1}(x) \quad \text{ for all } x \in W, i = 1,\dots,n-1.
\end{equation}
Finally, given a collection of balls $\calC$, we denote by $[\calC]$ the collection
of all maximal balls in $\calC$ with respect to the set inclusion,
i.e.,
\[
[ \calC ]=\{ U \in \calC: U \subseteq V \Rightarrow U = V \text{ for
  all } V \in \calC \}. 
\]
We will call the balls in $[ \calC ]$ \textbf{maximal}. It is
clear that if $\calU$ is a tower, then 
\[
\bigcup\limits_{U \in [\calU_n]} U = \bigcup\limits_{U \in \calU} U.
\]

\begin{lemma}
\label{l.maxballs}
Let $\prX$ be a metric space and $\varepsilon \in (0,1)$ be
arbitrary. Suppose that $\calU$ is a $\left( 1 + \frac 4 {\varepsilon}
\right)$-expanding tower over some subset $W \subseteq X$ of height
$n>1$. Then for all indexes $1 \leq i < j \leq n$ and for all balls $U
\in \calU_i$, $V \in [\calU_j]$ we have
\[
U \cap V \neq \varnothing \Rightarrow U \subseteq V^{\varepsilon}.
\]
\end{lemma}
\begin{proof}
Suppose that $U=\bl(x,s) \in
\calU_i$ and $V = \bl(y,t) \in [\calU_j]$ are some balls such that
\[
U \cap V \neq \varnothing
\]
and $U$ is \emph{not} a subset of $V^{\varepsilon}$. Then the ball $U$ does not
fit into the annulus $V^{\varepsilon} \setminus V$ of width
$\varepsilon t$, hence
\[
\varepsilon t \leq 2 s.
\] 
This implies that $s \geq \frac{\varepsilon t}{2}$, hence $\left( 2 +
  \frac 4 {\varepsilon} \right)s \geq 2s+2t$. We conclude that $V \subsetneq
U^{1 + \frac 4 {\varepsilon}}$, which is a contradiction since the ball $V$ is maximal
and the ball $U^{1 + \frac 4 {\varepsilon}}$ is contained in some ball
from $\calU_j$ by the definition of a tower.
\end{proof}

To simplify the presentation, we fix the following notation till the
end of this section. Let $(\prX, \mu)$ be an arbitrary metric measure space of exact polynomial growth
of degree $q \in \Rps$ and let $X_0 \subseteq X$ be a bounded
set. Define the constant
\[
C:=\frac{3^{q+1}}{3^{q+1}-1}.
\]
After these preparations, we can prove the so-called effective Vitali
covering theorem. The proof is based on the proof of the effective
Vitali covering lemma from \cite{mor2016}, where the results of
S. Kalikow and B. Weiss from \cite{kw1999} were generalized.

\begin{thm}[`Effective Vitali Covering']
\label{t.evc}
Let $\varepsilon \in (0,1)$ be small enough. Then there exist a number
$s_0 = s_0(\varepsilon,X_0)$ such that for a tower $\calU$ satisfying the
assertions
\begin{aufziii}
\item $\calU$ is $\left( 1 + \frac{36 q}{\varepsilon} \right)$-expanding;
\item $\core(\calU) \subseteq X_0$;
\item $\rad(\calU) \geq s_0$;
\item $\hght(\calU) \geq 1+\log_{C} \frac{2}{\varepsilon}$
\end{aufziii}
there exists a disjoint subcollection $\calV \subseteq \calU$ such
that
\begin{equation*}
\mu\left( \bigcup\limits_{V \in \calV} V\right) \geq (1 - \varepsilon) \mu\left( \bigcup\limits_{W \in \calU_1} W\right)
\end{equation*}
\end{thm}
\begin{proof}
First of all, we specify some of the parameters of the lemma. By
`small enough $\varepsilon$' we mean that $\varepsilon<\varepsilon_0$,
where $\varepsilon_0$ is given by Lemma \ref{l.smallexp}. Let $s'$ be the
parameter provided by Lemma \ref{l.vitcov} and $s=s(\varepsilon/2,X_0)$ be the
parameter provided by Lemma \ref{l.smallexp}. We define 
\[
s_0(\varepsilon,X_0):=\max(s',s(\varepsilon/2,X_0)).
\]
Let $n:=\hght(\calU)$ be height of
the tower $\calU$ and let $\delta:=\frac{\varepsilon}{9 q}$. For each $i=1,\dots,n$ let
\[
U_i := \bigcup\limits_{V \in \calU_i} V
\]
be the union of $i$-th level balls. The goal is to show that there
exists a collection of disjoin balls $\calV \subseteq \calU$ such that
\begin{equation}
\label{eq.evcgoal}
\mu\left( \bigcup\limits_{V \in \calV}  V^{\delta}\right) \geq (1 - \varepsilon/2) \mu\left( \bigcup\limits_{W \in \calU_1} W\right).
\end{equation}
Once such a collection is found, an application of Lemma
\ref{l.smallexp} would complete the proof of the theorem since
$\delta < \frac{1}{4q} \cdot \frac{\varepsilon} 2$.

The main idea of the proof is to cover a
positive fraction of $U_n$ by disjoint union of maximal $n$-level
balls from $\calU_n$ using Lemma \ref{l.vitcov}, then cover a positive
fraction of the remainder of $U_{n-1}$ by a disjoint union of
$(n-1)$-level balls and so on. We make this intuition formal below.

Let $\calC_n \subseteq \calU_n$ be the collection of disjoint balls,
which we obtain by applying Lemma \ref{l.vitcov} to the collection of
maximal $n$-level balls $[\calU_n]$. It follows that
\[
\mu\left( \bigcup\limits_{V \in \calC_n} V \right) \geq \frac{1}{3^{q+1}}
\mu\left( U_n \right)
\]
Let $S_n:=\bigsqcup\limits_{V \in \calC_n} V$. The computation above
shows that 
\begin{equation}
\label{eq.stepa}
S_n \text{ covers at least a } \frac 1 {3^{q+1}} \text{-fraction of } U_n
\end{equation}
and
\begin{equation}
\label{eq.stepb}
\mu( U_1 ) - \mu( S_n ) \leq \mu( U_1 ) - \frac{1}{3^{q+1}} \mu( U_1 ) =
C^{-1} \mu( U_1 ).
\end{equation}

In general, suppose by induction that the collections of disjoint
balls 
\[
\calC_n, \calC_{n-1},\dots, \calC_{n-k+1}
\]
with the respective unions
\[
S_n,S_{n-1},\dots,S_{n-k+1}
\]
have been constructed, where $k \geq
1$. Define the union of the corresponding $\delta$-expansions
\[
I_{k-1}:=\bigcup\limits_{V \in \calC_n \cup \dots \cup \calC_{n-k+1}} V^{\delta}.
\]
By induction, we assume that
\begin{equation}
\label{eq.indassert}
\mu \left( I_{k-1} \right) \geq (1-C^{-k+1}) \mu(U_1).
\end{equation}
Let
\begin{align*}
\widetilde \calC_{n-k}:=\{ V: \ &V \in [\calU_{n-k}] \text{ is a
                                           maximal } (n-k)-\text{level
                                           ball} \\
&\text{ such that } V \cap (S_n \cup S_{n-1} \cup \dots \cup S_{n-k+1}) = \varnothing\}
\end{align*}
be the collection of all maximal $(n-k)$-level balls disjoint from
$S_n \cup S_{n-1} \cup \dots \cup S_{n-k+1}$ and let $\widetilde
S_{n-k}$ be its union. We
apply Lemma \ref{l.vitcov} once to obtain a collection $\calC_{n-k}
\subseteq \widetilde \calC_{n-k}$ of pairwise disjoint balls such that
\[
\mu\left( \bigsqcup\limits_{V \in \calC_{n-k}} V \right) \geq \frac 1 {3^{q+1}}
\mu( \widetilde S_{n-k} )
\] 
and let $S_{n-k}:=\bigsqcup\limits_{V \in \calC_{n-k}} V$. We want to show that
\begin{equation}
\label{eq.snest}
\mu\left( \bigcup\limits_{V \in \calC_n \cup \dots \cup \calC_{n-k+1}} V^{\delta} \cup S_{n-k} \right) \geq (1-C^{-k}) \mu(U_1)
\end{equation}
and to do so it suffices to prove that
\begin{equation}
\label{eq.snm1}
\mu\left( \bigcup\limits_{V \in \calC_n \cup \dots \cup \calC_{n-k+1}}
  V^{\delta} \cup S_{n-k} \right) \geq \mu(I_{k-1}) + \frac 1 {3^{q+1}} \mu( U_{n-k}
  \setminus I_{k-1}),
\end{equation}
due to the inductive assertion above (Equation \eqref{eq.indassert}).
We decompose the set $U_{n-k} \setminus I_{k-1}$ as follows:
\[
U_{n-k} \setminus I_{k-1} = \widetilde S_{n-k} \sqcup \left( U_{n-k}
  \setminus (I_{k-1} \cup \widetilde S_{n-k})\right).
\]
The part $U_{n-k}
  \setminus (I_{k-1} \cup \widetilde S_{n-k})$ is covered by the $(n-k)$-level balls
  intersecting $S_n \cup S_{n-1} \cup \dots \cup S_{n-k+1}$. Hence, due to Lemma \ref{l.maxballs}, the set $U_{n-k}
  \setminus (I_{k-1} \cup \widetilde S_{n-k})$ is covered by the $\delta$-expansions of
  balls in $\calC_n \cup \dots \cup \calC_{n-k+1}$. Next, $S_{n-k}$
  covers at least a $\frac 1 {3^{q+1}}$-fraction of the set
  $\widetilde S_{n-k}$. It follows that the set 
\[
\bigcup\limits_{V \in \calC_n \cup \dots \cup \calC_{n-k+1}}
  V^{\delta} \cup S_{n-k}
\]
covers the set $I_{k-1}$ and at least a $\frac 1
  {3^{q+1}}$-fraction of the set $U_{n-k} \setminus I_{k-1}$. Thus we have proved
  inequalities \eqref{eq.snm1} and \eqref{eq.snest}. 

It is obvious that one can continue in this way down to the $1$-st
level balls. This would yield a collection of maximal balls
\[
\calC:=\bigcup\limits_{i=1}^n \calC_i
\]
so that the measure of the union of $\delta$-expansions of balls in
$\calC$ is at least $\left( 1- C^{-n}\right)$ times the measure of
$U_1$. Since $n \geq 1+ \log_C\frac{2}{\varepsilon}$, we deduce that
the proof of the inequality \eqref{eq.evcgoal} is complete. 
\end{proof}

\begin{rem}
\label{r.smallbdrunexgr3}
We note that for metric measure spaces of
\emph{uniformly} exact polynomial growth the assumption in Theorem
\ref{t.evc} that
\[
\core(\calU) \subseteq X_0
\]
for a bounded set $X_0$ can be dropped. In this case the parameter
$s_0$ depends on $\varepsilon$ and the space $(\prX,\mu)$ only.
\end{rem}

Theorem \ref{t.evc} will play an essential role in the proof of the
following result. For convenience, we fix some additional notation till the end of the
section. We let
$0< \varepsilon < \varepsilon_0$ be arbitrary, and define
\[
s_1 := \max( s_0(\varepsilon,X_0), s_*(X_0))
\]
to be the maximum of the corresponding parameters provided by Theorem
\ref{t.evc} and Lemma \ref{l.distpoint}. Fix $L \in \Zp$.

Assume that $\calU$ and $\calV$
are two centered towers over the set $X_0 \subseteq X$ of height $L+1$ such that 
\[
\core(\calU), \core(\calV) \subseteq X_0,
\]
\[
\rad(\calU), \rad(\calV) \geq s_1
\]
and
\[
U_1(x) \subseteq V_1(x) \subseteq U_2(x) \subseteq V_2(x) \subseteq \dots \subseteq
U_{L+1}(x) \subseteq V_{L+1}(x)
\]
for all points $x \in X_0$. Let
\[
\Delta:=\max(144q,4).
\]
Suppose that
\begin{equation}
\label{eq.vugrowth}
U_i^{\left( 1+ \Delta / \varepsilon \right)}(x) \subseteq V_i(x)
\end{equation}
for all indexes $i=1,\dots,L+1$ and all points $x \in X_0$. It is clear
that, since the towers are centered, this implies that the towers $\calU,\calV$ are $\left(
  1+\frac{\Delta}{\varepsilon}\right)$-expanding.

\begin{lemma}
\label{l.fill}
Suppose that $L \geq 1+ 2 \log_C \frac{2}{\varepsilon}$. Then either
there exists a ball $W \in \calV_{L+1}$ which can be $4 \varepsilon$-filled
by a disjoint subcollection of $\calU$, or, otherwise, the inequality 
\begin{equation}
\label{eq.epsfill}
\mu \left( \bigcup\limits_{V \in \calV} V \right) \geq \left( 1
  + \frac{\varepsilon}{3^{q+1}}\right) \mu \left( \bigcup\limits_{U \in
    \calU_1} U \right)
\end{equation}
holds.
\end{lemma}
\begin{proof}
First of all, we apply Lemma \ref{l.vitcov} to the collection
$[\calV_{L+1}]$, obtaining a collection of maximal disjoint balls $\calW \subseteq [
\calV_{L+1} ]$ such
that
\[
\mu\left( \bigcup\limits_{ W \in \calW} W \right) \geq \frac{1}{3^{q+1}}
\mu\left( \bigcup\limits_{V \in [\calV]} V \right) \geq
\frac{1}{3^{q+1}} \mu \left( \bigcup\limits_{U \in \calU_1} U\right). 
\]
There are two further possibilities. First, suppose that for every $W \in
\calW$ we have
\[
\mu\left( W \setminus \bigcup\limits_{U \in \calU_{\lfloor L/2 \rfloor}} U \right) \geq
\varepsilon \mu \left( W \right).
\]
Let $D:= \bigcup\limits_{U \in \calU_{ \lfloor L/2 \rfloor }} U$. Then
$\bigcup\limits_{U \in \calU_1} U \subseteq D$, and so
\begin{align*}
\mu\left( \bigcup\limits_{V \in \calV} V \right) &\geq \mu \left(
  \bigcup\limits_{U \in \calU_1} U \right) + \mu \left(
  \bigcup\limits_{W \in \calW} W \setminus D \right) = \\
&= \mu \left( \bigcup\limits_{U \in \calU_1} U\right) +
  \sum\limits_{W \in \calW} \mu \left( W \setminus D \right) \geq \\
&\geq \left( 1+ \frac{\varepsilon}{3^{q+1}} \right) \mu \left(
  \bigcup\limits_{U \in \calU_1} U\right).
\end{align*}
This is precisely inequality \eqref{eq.epsfill} in the statement of the
lemma. Otherwise, suppose that the ball $W \in \calW$ is such that
\[
\mu\left( W \setminus \bigcup\limits_{U \in \calU_{ \lfloor L/2 \rfloor }} U \right) <
\varepsilon \mu \left( W \right).
\]
We intend to prove that $W$ can be $4 \varepsilon$-filled by a disjoint
subcollection of $\calU$. Define the
collections of balls
\[
\calY:= \{ U_i(x): 1 \leq i \leq \lfloor L/2 \rfloor , x \in X_0
\text{ are such that }
U_{\lfloor L/2 \rfloor }(x) \cap W \neq \varnothing \}
\]
and
\[
\calZ:= \{ U_i(x): 1 \leq i \leq L, x \in X_0 \text{ are such that }
U_L(x) \subseteq W \}.
\]
Observe that $\calZ$ is in fact a \emph{tower}. Furthermore, 
\begin{equation}
\label{eq.ycond}
\mu\left( W \cap \bigcup\limits_{Y \in \calY} Y \right) \geq (1- \varepsilon) \mu \left( W \right).
\end{equation}
We will show that
\begin{equation}
\label{eq.zfillg}
\mu\left( W \cap \bigcup\limits_{Z \in \calZ_{\leq \lfloor L/2 \rfloor
    }} Z \right)
\geq (1- 3 \varepsilon) \mu \left( W \right).
\end{equation}
Assume that this does not hold. Lemma \ref{l.smallbdr} implies that
\[
\mu(\kint{\varepsilon/4q}{W}) \geq (1-\varepsilon) \mu(W).
\]
Then there exists a subset $W_0 \subseteq \kint{\varepsilon/4q}{W}$
such that 
\[
W_0 \cap \bigcup\limits_{Z \in \calZ_{\leq \lfloor L/2 \rfloor
    }} Z = \varnothing
\]
and $\mu(W_0) \geq
2 \varepsilon \mu(W)$. Condition \eqref{eq.ycond} implies that there
exists a ball $U_i(x) \in \calY \setminus \calZ_{\leq \lfloor L/2
  \rfloor }$ such that
$U_i(x) \cap W_0 \neq \varnothing$. Let $z_1 \in U_i(x) \cap W_0$. Since $U_i(x)
\notin \calZ_{\leq \lfloor L/2 \rfloor }$, we know that $U_L(x) \cap W^c \neq
\varnothing$. Let $z_2 \in U_L(x) \cap W^c$, then, clearly,
$z_1,z_2 \in U_L(x)$.

Suppose for the moment that $W = \bl(y,r)$ for some $y \in X_0$
and some $r > 0$. Since $W_0 \subseteq \kint{\varepsilon/4q}{W}$, we have
\[
d(z_1,z_2) \geq \frac{\varepsilon r}{4 q},
\]
and hence the radius of the ball $U_L(x)$ is greater or equal
than $\frac{\varepsilon r}{8 q}$. If $V_L(x) = \bl(x,r')$ for some $r'>0$, then, since $U_L^{\left( 1+ 144
    q / \varepsilon \right)}(x) \subseteq V_L(x)$, we conclude using
Lemma \ref{l.distpoint} that
\[
r' \geq \left( 2+\frac{144 q}{\varepsilon}\right)\cdot \frac{1}{3}
\cdot \frac{\varepsilon r}{8 q} \geq 6 r.
\]
The intersection $V_L(x) \cap W$ is nonempty, since it contains point
$z_1$. The tower $\calV$ is $(1+\frac{4}{\varepsilon})$-expanding as
well, hence $V_L(x) \subseteq W^{\varepsilon}$ due to Lemma \ref{l.maxballs}. This contradicts to
the estimate $r' \geq 6r$ above. Thus we have proved the estimate \eqref{eq.zfillg}.

Now consider the tower $\widetilde \calZ := \calZ_{\geq \lfloor L/2
  \rfloor}$, which has the height of at least $1+\log_C\frac{2}{\varepsilon}$. Due to the
estimate \eqref{eq.zfillg}, the first level balls of the tower
$\widetilde \calZ$ cover at least a $(1-3 \varepsilon)$-fraction of
$W$. We apply Theorem \ref{t.evc} to the tower $\widetilde \calZ$ and
obtain a disjoint collection of balls $\calZ'$, which are all contained in $W$, such that
\[
\mu\left( \bigcup\limits_{Z \in \calZ'} Z\right) \geq (1 -
\varepsilon) \mu\left( \bigcup\limits_{Z \in \widetilde \calZ_1}
  Z\right) \geq (1-4\varepsilon) \mu(W).
\]
This shows that the ball $W$ can be $4 \varepsilon$-filled by a
disjoint collection of balls from $\calU$, and the proof is complete.
\end{proof}

For the convenience of the reader we recollect all assumptions in the
statement of the following main theorem of the section. The proof is
a minor adaptation of the one in \cite{hochman2009}. We remind the reader that the constants $C,\Delta$ are defined as
\[
C:=\frac{3^{q+1}}{3^{q+1}-1}, \quad \Delta:=\max(144q,4),
\]
and we define a new constant $K:=\lceil 2
\log_C{\frac{4}{\varepsilon}} \rceil$.
\begin{thm}
\label{t.expgr}
Let $(\prX, \mu)$ be a metric measure space of exact polynomial growth
of degree $q \in \Rps$ and let $X_0 \subseteq X$ be a bounded
set. Then for all $\varepsilon>0$ small enough there exists
$s_1:=s_1(X_0,\varepsilon)$ such that the following holds. Let $L \geq
2 \log_C(4/\varepsilon)$ be an arbitrary integer. Assume that $\calU$ and $\calV$
are centered towers over $X_0$ of height $L+1$
such that 
\[
\core(\calU), \core(\calV) \subseteq X_0,
\]
\[
\rad(\calU), \rad(\calV) \geq s_1
\]
and
\[
U_1(x) \subseteq V_1(x) \subseteq U_2(x) \subseteq V_2(x) \subseteq \dots \subseteq
U_{L+1}(x) \subseteq V_{L+1}(x) \quad \text{ for all } x\in X_0.
\]
Suppose that
\begin{equation}
\label{eq.vugrowth}
U_i^{\left( 1+ \Delta / \varepsilon \right)}(x) \subseteq V_i(x) \quad
\text{ for all } x\in X_0, i=1,\dots,L+1
\end{equation}
and that no ball $V \in \calV$ can be $4\varepsilon$-filled by a
disjoint subcollection of $\calU$. Then
\[
\mu \left( \bigcup\limits_{U \in \calU_1} U\right) \leq \left(
  1+\frac{\varepsilon}{3^{q+1}}\right)^{-\lfloor L/ K  \rfloor} \mu \left( \bigcup\limits_{V \in
    \calV_{L+1}} V\right).
\]
\end{thm}
\begin{proof}
We assume that $0<\varepsilon<\varepsilon_0$ (with $\varepsilon_0$
determined in Theorem \ref{t.evc}). Consider the case $L:=K$ first. Then the statement of
the theorem is precisely Lemma \ref{l.fill}, since $\lceil 2
\log_C{\frac{4}{\varepsilon}} \rceil \geq 1+2\log_C{\frac{2}{\varepsilon}}$. In general, we suppose
that $iK \leq L < (i+1)K$ for some uniquely determined $i \in \N$
and proceed by induction on $i$. So suppose that we have proved the
statement for $i \geq 1$ and want to proceed to $i+1$. Let $\calU,\calV$ be
two arbitrary towers of height $n$ such that $1+(i+1)K \leq n < 1+(i+2)K$, satisfying the assertions of the theorem. Then
\[
\widetilde \calU:=\calU_{\leq n-K}, \quad \widetilde
\calV:=\calV_{\leq n- K}
\]
are two towers of height $n- K$, satisfying the assertions of the
theorem. The inductive assumption implies that
\[
\mu \left( \bigcup\limits_{U \in \widetilde \calU_1} U\right) \leq \left(
  1+\frac{\varepsilon}{3^{q+1}}\right)^{-i} \mu \left( \bigcup\limits_{V \in
    \widetilde \calV_{n-K}} V\right) \leq \left(
  1+\frac{\varepsilon}{3^{q+1}}\right)^{-i} \mu \left( \bigcup\limits_{U \in
    \calU_{n-K+1}} U\right).
\]
Now observe that the towers $\calV':=\calV_{\geq n- K+1},
\calU':=\calU_{\geq n - K+1}$ of height $K$ satisfy the assertions of
the theorem as well. Thus we can apply the base case of the induction,
and so
\[
\mu \left( \bigcup\limits_{U \in
    \calU_{n-K+1}} U\right) \leq \left(
  1+\frac{\varepsilon}{3^{q+1}}\right)^{-1} \mu \left( \bigcup\limits_{V \in
     \calV_{n}} V\right). 
\]
Combining this with the the previous inequality, we deduce that
\[
\mu \left( \bigcup\limits_{U \in \calU_1} U\right) \leq \left(
  1+\frac{\varepsilon}{3^{q+1}}\right)^{-i-1} \mu \left( \bigcup\limits_{V \in
     \calV_{n}} V\right),
\]
and the proof is complete.
\end{proof}
\begin{rem}
\label{r.smallbdrunexgr4}
Similar to the previous remarks, we observe that for metric measure spaces of
\emph{uniformly} exact polynomial growth the assumption in Theorem
\ref{t.expgr} that
\[
\core(\calU),\core(\calV) \subseteq X_0
\]
for a bounded set $X_0$ can be dropped. In this case the parameter
$s_1$ depends on $\varepsilon$ and the space $(\prX,\mu)$ only.
\end{rem}

\section{Hochman's and Kingman's Theorems}
\label{s.upcr}
\subsection{Hochman's Upcrossing Theorem}
\label{ss.hochman}
The goal of this section is to prove the following theorem, which we
stated in the introduction.
\begin{thm}
\label{t.hutp}
Let $\Gamma$ be a group of polynomial growth. For all intervals
$(\alpha,\beta) \subset \R$ and all $\delta>0$ small enough there
exist constants $c>0$ and $\rho \in (0,1)$ such that the following
holds. If $(S_{\bl(g,r)})$ is a stationary process, then for
all $k \geq 1$
\begin{align*}
&\bbP(\{ x: (S_{\bl(i)}(x))_{i \geq 1} \in \flc{\alpha}{\beta}{k}  \})
  \leq \\
 &\leq c \rho^k + \bbP \left(\left\{ x: 
\begin{array}{c}
\exists n>k \text{ s.t. }
  S_{\bl(n)}(x)>\beta \text{ and } \bl(n) \text{ can be } \delta-\text{filled}\\
 \text{by disjoint balls } V_1,\dots,V_m \text{ s.t. } \forall i \
  S_{V_i}(x)<\alpha
\end{array} \right\} \right).
\end{align*}
\end{thm}

\noindent We proceed to the proof. Without loss of generality we
assume that $\Gamma$ is group of polynomial growth of degree $q \geq
1$. For all $k,l \in \N$ and all $g \in \Gamma$ we define the events
\[
Q_{g,l}^k:=\{ x: (S_{\bl(g,i)}(x))_{i = 1}^l \in \flc{\alpha}{\beta}{k}  \}
\]
and
\[
R_{g}^k:= \left\{ x: 
\begin{array}{c}
\exists n>k \text{ s.t. }
  S_{\bl(g,n)}(x)>\beta \text{ and } \bl(g,n) \text{ can be } \delta-\text{filled}\\
 \text{by disjoint balls } V_1,\dots,V_m \text{ s.t. } \forall i \
  S_{V_i}(x)<\alpha
\end{array} \right\}.
\]
Let $Q_g^k:=\bigcup\limits_{l \geq 1} Q_{g,l}^k$ be the increasing
union of the events $Q_{g,l}^k$ for all $l \geq 1$. The goal is to show existence of universal constants
$c>0,\rho \in (0,1)$ such that that
\[
\bbP(Q_{\ue}^k) \leq c \rho^k + \bbP(R_{\ue}^k) \quad \text{ for all } k \geq 1,
\]
and it is clear that to do so it suffices to prove that for all $l$ we have
\[
\bbP(Q_{\ue,l}^k \setminus R_{\ue}^k ) \leq c \rho^k \quad \text{ for all } k \geq 1.
\]

The main idea of the proof is to use a `transference principle'. For integers $l,M \in \N$ and a point $x \in
X$ we define the set
\[
E_{M,l,x}^k:= \{ g: \ x \in Q_{g,l}^k \setminus R_{g}^k \text{  and  } \| g \| \leq M \} \subseteq \bl(M).
\]
The lemma below tells us essentially that each universal upper bound on the
density of $E_{M,l,x}^k$ in $\bl(M)$ bounds the probability of $Q_{\ue,l}^k \setminus R_{\ue}^k$ from
above as well.
\begin{lemma}[`Transference principle']
\label{l.caldtrans}
Suppose that for a given constant $t \in \Rp$ the following holds:
there is some $M_0 \in \N$ such that for all $M \geq M_0, k \geq 1$ and for $\bbP$-almost all $x \in
X$ we have
\begin{equation}
\label{eq.transf}
 \cntm{E_{M,l,x}^k} \leq t \cntm{\bl(M)} + o_{x,l}(\cntm{\bl(M)}),
\end{equation}
where $o_{x,l}(\cntm{\bl(M)}) / \cntm{\bl(M)}$ converges to $0$ uniformly in $x$. Then
\[
\mu(Q_{\ue,l}^k \setminus R_{\ue}^k) \leq t
\]
for all $k \geq 1$.
\end{lemma}
\begin{proof}
Indeed, since the process is stationary, we have
\[
\sum\limits_{g \in \bl(M)} \int\limits_{\prX} \indi{Q_{g,l}^k \setminus R_{g}^k}(x) d \mu = \cntm{\bl(M)} \mu(Q_{\ue,l}^k \setminus R_{\ue}^k)
\]
for all $M \geq 1$. Then
\begin{align*}
\mu(Q_{\ue,l}^k \setminus R_{\ue}^k) &= \int\limits_{\prX} \left( \frac 1 {\cntm{\bl(M)}} \sum\limits_{g \in
      \bl(M)} \indi{Q_{g,l}^k \setminus R_{g}^k} (x) \right) d \mu \leq \\
&\leq \int\limits_{\prX} \left( \frac{\cntm{E_{M,l,x}}}{\cntm{\bl(L)}}
  \right) d \mu,
\end{align*}
and the proof is complete since $M$ can be arbitrarily large.
\end{proof}

The goal now is to derive estimate \eqref{eq.transf} such that the
assertions of Lemma \ref{l.caldtrans} hold, i.e., to prove
that for some universal constants $c > 0, \rho \in (0,1)$ we have that 
\[
\cntm{E_{M,l,x}^k} \leq c \rho^k \cntm{\bl(M)} + o_{x,l}(\cntm{\bl(M)})
\]
holds for all $M$ large enough, all $k$ and almost all $x$. Observe that it
suffices to do so for all $k$ `large enough' and change the universal
constant $c$ if necessary. Fix an arbitrary $x
\in X$. 

For each $g \in E_{M,l,x}^k$ there exists a sequence of balls $U_1(g)
\subsetneq V_1(g) \subsetneq \dots \subsetneq U_k(g) \subsetneq
V_k(g)$ with center $g$ such that for every 
\[
S_{U_i}(x)<\alpha \text{ and } S_{V_i(g)}(x)>\beta \quad \text{ for
  all } i=1,\dots,k.
\]
It is clear that for every $g \in E_{M,l,x}^k$ and every index $i=1,\dots,k-1$ we have
\begin{equation}
\label{eq.szgrwth}
\cntm{V_{i+1}(g)} > \cntm{U_{i+1}(g)}>\cntm{V_{i}(g)} >  \cntm{U_{i}(g)}.
\end{equation}
Since $\Gamma$ is a group of polynomial growth endowed with a word
metric and the balls $U_{i+1}(g)$,$U_{i}(g)$,$V_{i+1}(g)$ and $V_i(g)$
above are centered at $g$,
condition \eqref{eq.szgrwth} implies that the radius of the ball
$U_{i+1}(g)$ is greater or equal than the radius of the ball $U_i(g)$
plus $2$. Similarly, the radius of $V_{i+1}(g)$ is at least the radius
of $V_i(g)$ plus $2$.

Let $s_1:=\max(s_1(\delta/4), s_!)$ be the maximum of the constants given by Theorem
\ref{t.expgr} and Lemma \ref{l.measvsrad}. We skip the first $k':=\max(\lceil s_1
\rceil, \lceil k/2 \rceil)$ upcrossings to ensure that the radius of
the balls $U_i(g), V_i(g)$ is not smaller than $s_1$ and $k$. Observe that
$\frac{\cntm{V_i(g)}}{\cntm{U_i(g)}} > 1+\delta$ for all indexes $i$
and all $g \in E_{M,l,x}^k$. Indeed, suppose on the contrary that $\frac{\cntm{V_i(g)}}{\cntm{U_i(g)}} \leq 1+\delta$ for
some index $i$ and some $g \in E_{M,l,x}^k$. Then $U_i(g)$
$\delta$-fills $V_i(g)$ and the radius of $V_i(g)$ is at least $k$,
thus leading to a contradiction, since $x \notin R_g^k$ by definition of $E_{M,l,x}^k$. Lemma \ref{l.measvsrad} now implies that
if $V_{i}(g) = \bl(g,r_1)$ and $U_i(g)=\bl(g,r_2)$, then
\[
\frac{r_1}{r_2} > \left( \frac{1+\delta}{1+2\delta/3} \right)^{1/q}.
\]

Let 
\[
D:=q \left\lceil
  \frac{\log(2+576q/\delta)}{\log(1+\delta)-\log(1+2\delta/3)}
\right\rceil + 1 \in \N.
\]
We define towers $\widetilde \calU, \widetilde \calV$ over the set
$E_{M,l,x}^k$ of height $L=\lfloor (k-k')/D \rfloor$ by setting
\[
\widetilde U_i(g):=U_{D (i-1)+1}(g), \quad \widetilde V_i(g):=V_{D (i-1)+1}(g)
\]
for all indexes $i = 1,\dots,L$ and all $g \in
E_{M,l,x}^k$. 
Recall that
\[
C=\frac{3^{q+1}}{3^{q+1}-1}, \quad K = \lceil 2 \log_C\frac{16}{\delta}\rceil.
\]
Theorem \ref{t.expgr} applied to the towers $\widetilde \calU,
\widetilde \calV$ now implies that
\[
\cntm{E_{M,l,x}^k} \leq \left( 1+\frac{\delta}{4 \cdot
    3^{q+1}}\right)^{-\lfloor L / K \rfloor} \cntm{\bl(M+l)},
\]
since $\bigcup\limits_{V \in \widetilde \calV} V \subseteq
\bl(M+l)$. It follows that there exist universal constants $c>0, \rho \in (0,1)$
such that
\[
\cntm{E_{M,l,x}^k} \leq c \rho^k \cntm{\bl(M)} + o_{x,l}(\cntm{\bl(M)})
\]
for all $k$ large enough, and the proof is complete.

\subsection{Kingman's Subadditive Ergodic Theorem}
\label{ss.kingman}
As one of the applications of Theorem \ref{t.hutp} we show how the
proof of exponential decay in Kingman's subadditive ergodic theorem
from \cite{hochman2009} carries over to groups of polynomial growth.

A stochastic process $(S_{\bl(g,r)})$ on a group of polynomial growth $\Gamma$ is called \emph{subadditive} if
for every ball $V \subset \Gamma$ and for every finite collection of disjoint
balls $V_1,V_2,\dots,V_n$ such that $V = \bigcup\limits_{i=1}^n V_i$
we have
\[
S_{V} \leq \sum\limits_{i=1}^n S_{V_i}.
\]

\begin{thm}
Let $(S_{\bl(g,r)})$ be a stationary subadditive process on a group of
polynomial growth $\Gamma$ such that for
some constant $C>0$ we have
$S_{\bl(\ue,0)} \leq C$ almost surely. Then for every interval
$(\alpha,\beta)$ there exist constants $c>0,\rho \in (0,1)$, depending only on
$\alpha,\beta, \Gamma$ and $C$, such that
\[
\bbP\left(\left\{ x: \left( \frac{1}{\cntm{\bl(i)}} S_{\bl(i)} \right)_{i
    \geq 0} \in \flc{\alpha}{\beta}{k} \right\} \right)
\leq c \rho^k
\]
for all $k \geq 1$.
\end{thm}
\begin{proof}
First of all, by stationarity and subadditivity of the process
$(S_{\bl(g,r)})$ we have
\[
\frac 1 {\cntm{\bl(g,r)}} S_{\bl(g,r)}(x) \leq C \quad \text{ for all } g
\in \Gamma, r \geq 0
\]
for all $x$ in a some subset $X_0 \subseteq X$ of full measure. This implies that it suffices to consider only the
case $C > \alpha$. Furthermore, by replacing if necessary the interval
$(\alpha,\beta)$ with a smaller subinterval, it suffices to prove the
theorem for all $(\alpha,\beta)$ such that
\[
\delta:=\frac{\beta - \alpha}{C} < \delta_0,
\]
where $\delta_0$ is the `small enough $\delta$' given by Theorem
\ref{t.hutp}. We apply Theorem \ref{t.hutp} with this value of
$\delta$ to the stationary process
\[
\left( \frac{1}{\cntm{\bl(g,r)}} S_{\bl(g,r)} \right),
\]
and it only remains to show that the event
\[
R:=\left\{ x \in X_0: 
\begin{array}{c}
\exists n>k \text{ s.t. }
  \frac{1}{\cntm{\bl(n)}} S_{\bl(n)}(x)>\beta \text{ and } \bl(n) \text{ can be  } \delta-\text{filled}\\
 \text{by disjoint balls } V_1,\dots,V_m \text{ s.t. } \forall i \
  \frac{1}{\cntm{V_i}} S_{V_i}(x)<\alpha
\end{array} \right\}
\] 
is empty. Suppose the contrary and pick $x \in R$. If a ball $\bl(n)$
such that $\frac{1}{\cntm{\bl(n)}} S_{\bl(n)}(x)>\beta$ is $\delta$-filled by the
balls $V_1,\dots,V_m$ such that $S_{V_i}(x)<\alpha \cntm{V_i}$ for all $i$, then
via subadditivity of the process we obtain the inequality
\begin{align*}
\frac{1}{\cntm{\bl(n)}} S_{\bl(n)}(x) &\leq \sum\limits_{i=1}^m
\frac{1}{\cntm{\bl(n)}} S_{V_i}(x)+\frac{1}{\cntm{\bl(n)}}\sum\limits_{g \in \bl(n) \setminus \bigcup
  V_i} S_{\bl(g,0)}(x) \leq \\
&< \sum\limits_{i=1}^m
\frac{\alpha \cntm{V_i}}{\cntm{\bl(n)}} + \frac{\cntm{\bl(n) \setminus \bigcup
  V_i}}{\cntm{\bl(n)}} C \leq \alpha + \delta C.
\end{align*}
A direct computation shows that $ \alpha + \delta C
= \beta$, contradiction.
\end{proof}

\bigskip

{\bf Acknowledgements.}\
This research was carried out during the author's PhD studies in TU Delft
under the supervision of Markus Haase. I would like to thank him for
his support. The author also kindly acknowledges the financial support from Delft Institute of Applied Mathematics.

\printbibliography[]
\end{document}